\documentclass[12pt,a4paper]{amsart}
\usepackage[utf8]{inputenc}
\usepackage{hyperref}
\usepackage{indentfirst}
\usepackage{graphics}
\usepackage{amssymb}
\usepackage{amsmath}
\usepackage{verbatim}
\usepackage[margin=2.8cm]{geometry}
\usepackage{epstopdf} 
\usepackage{cite}
\usepackage{xcolor}

\usepackage{caption}
\usepackage{subcaption}

\usepackage[pdftex]{graphicx}

\newcommand{\zz}{\ensuremath{\mathbb{Z}}}
\newcommand{\qq}{\ensuremath{\mathbb{Q}}}

\newcommand{\cc}{\ensuremath{\mathbb{C}}}

\newcommand{\pp}{\ensuremath{\mathfrak{p}}}

\newcommand{\ord}{\textnormal{ord}}
\newcommand{\pideal}{\ensuremath{\mathfrak{p}}}
\newcommand{\Pideal}{\ensuremath{\mathfrak{P}}}

\graphicspath{ {./Desktop/ } }

\theoremstyle{plain}
\newtheorem{Th}{Theorem}[section]
\newtheorem{Lemma}[Th]{Lemma}
\newtheorem{Cor}[Th]{Corollary}

\theoremstyle{definition}

\newtheorem{?}[Th]{Problem}
\newtheorem{Ex}[Th]{Example}

\newcommand{\SL}{\operatorname{SL}}

\begin{document}
	
	\title{Divisors of Modular Parametrizations of Elliptic Curves}

	\author{Michael Griffin and Jonathan Hales}
	
	\address{Brigham Young University \\ Department of Mathematics \\
		Provo, UT 84602 United States of America }
	
	\email{Mjgriffin@math.byu.edu , jhales@mathematics.byu.edu}
	

	\begin{abstract} 
		The modularity theorem implies that for every elliptic curve $E /\mathbb{Q}$ there exist rational maps from the modular curve $X_0(N)$ to $E$, where $N$ is the conductor of $E$. These maps may be expressed in terms of pairs of modular functions $X(z)$ and $Y(z)$ where $X(z)$ and $Y(z)$ satisfy the Weierstrass equation for $E$ as well as a certain differential equation. Using these two relations, a recursive algorithm can be used to calculate the $q$ - expansions of these parametrizations at any cusp. 
		 Using these functions, we determine the divisor of the parametrization and the preimage of rational points on $E$. We give a sufficient condition for when these preimages correspond to CM points on $X_0(N)$. We also examine a connection between the algebras generated by these functions for related elliptic curves, and describe sufficient conditions to determine congruences in the $q$-expansions of these objects.
	\end{abstract}
	
	\maketitle
	
	\section{Introduction and statement of results} 
	The modularity theorem \cite{Wiles,Modularity} guarantees that for every elliptic curve $E$ of conductor $N$ there exists a weight $2$ newform $f_E$ of level $N$ with Fourier coefficients in $\zz$. The Eichler integral of $f_E$ (see (\ref{defEichler})) and the Weierstrass $\wp$-function together give a rational map from the modular curve $X_0(N)$ to the coordinates of some model of $E.$ 
This parametrization has singularities wherever the value of the Eichler integral is in the period lattice.
 Kodgis \cite{Kodgis} showed computationally that many of the zeros of the Eichler integral occur at CM points. Peluse \cite{Peluse} later proved several general cases confirming many of these conjectured zeros using the theory of Hecke operators and Atkin–Lehner involutions.
	
	In \cite{AGOR}, the authors use the modular parametrization of an elliptic curve to give a harmonic Maass form of weight $3/2$ whose Fourier coefficients encode the vanishing of central $L$-values and $L$-derivatives of quadratic twists of the curve. The Birch and Swinerton-Dyer conjecture asserts that the order of vanishing of the central $L$-value of an elliptic curve is the rank of the curve. Kolyvagin\cite{Kolyvagin} confirmed this conjecture if the order of vanishing is less than $2$. Unfortunately, the result of \cite{AGOR} is only fully constructive if the modular parametrization is holomorphic on the upper half plane. Otherwise we must remove the singularities, a task which is difficult without knowledge of their locations.

	\par For a modular function $F$ for some subgroup $\Gamma$ of $\SL_{2}(\zz)$, 
	 we consider the \emph{modular polynomial} of $F$ 
	
	\begin{equation}\label{ModPoly}
	\Phi_F(x) := \prod_{\gamma \in \Gamma \backslash SL_{2}(\zz)}\Big{(}x - F(\gamma z)\Big{)} = \sum A_i(z)x^{i}.
	\end{equation}
	
	\noindent One of our goals is to calculate the minimal divisor of $(\ref{ModPoly})$ for $F$ which are rational in terms of the coordinates functions $(X(z),Y(z))$ of a given modular parametrization of $E$, chosen so as to have poles at the divisor of the parametrization.  
We may calculate the divisor by calculating the divisor of the coefficient functions $A_i(z)$. 
 In order to calculate the product in $(\ref{ModPoly})$ we need the expansion of $F$ at each of the cusps of $\Gamma$. Algorithms for calculating the coefficients of $X(z)$ and $Y(z)$ at the cusp infinity are described by Cremona \cite{Cremona}, and we include a variation of that method that allows for the computation of coefficients at any cusp.
	
	\begin{Ex}
		For the elliptic curve \begin{equation}
		E: y^2+y = x^3-x^2 -10x-20 \tag{11a1}
		\end{equation} 
		one can calculate that $E$ has $(5,5)$ and $(5,-6)$ as points of order $5$. If we set $F(z) = \left( X(z)-5 \right)^{-1}$, then $F(z)$ has zeros only when $z$ is an element of the complex lattice associated to $E$, and poles only when $z$ is mapped to one of these $5$-torsion points. Computing the divisor of $\Phi_F(X)$, we find that 
		\[X(z) = 5 \quad \implies \quad (j(z)+24729001)(j(z)+32768)=0.\]
		If $z = \frac{1+\sqrt{-11}}{2}$, then  $j(z) = -32768$. Since $j(z)$ is invariant under the action of $\SL_2(\zz)$ while $F$ is only $\Gamma_0(11)$ invariant, we look at the $\Gamma_0(11)\backslash \SL_2(\zz)$ orbit of $z$ to find
		
		\[z_0 = \frac{-11 +\sqrt{-11}}{55} \quad \implies \quad \big{(}X(z_0),Y(z_0)\big{)} = (5,5).\]
		Thus the point $z_0$ is a preimage of the rational  point $(5,5)$, and is a CM point on $X_0(11)$.  
	\end{Ex} 
	The points of $X_0(N)$ are in correspondence with pairs $(e,c)$ where $e$ is an elliptic curve and $c\subset e$ is a cyclic subgroup of order $N$ (See Appendix C.13 of \cite{Silverman}). Using this description, we give a sufficient condition for when a point $\mathcal{P}$ lying above a rational point $P$ on $E$ is a CM point. The proof is given in section $3$.

	\begin{Th}
		Fix an elliptic curve $E/\qq$ of conductor $N$ and $P$ a point on $E$. Let $\mathcal{P}$ a point on $X_0(N)$ that maps to $P$ under some modular parametrization, and which is in correspondence to the pair $(e,c)$ where $e$ is an elliptic curve over a number field $K$. For each $m \mid \mid N$, either $e$ admits an $m$-isogeny defined over $K$ or $e$ has CM by an order of discriminant $D$ where $0 \leq -D \leq 4m$ and $D$ is a square $\pmod{4m}$.
	\end{Th}

	\par In section $4$ we consider the question, given an elliptic curve $E$, when are the coefficients of these parametrizations contained in some prime ideal $\mathfrak{p}$ of a number ring $\mathcal{O}$? One sufficient condition we give is that the elliptic curves are isogenous, and have congruent coefficients mod $p$ for some prime $p$ lying below $\mathfrak{p}$. Another sufficient condition we provide is a bound similar to Sturm's bound that implies that every coefficient of the parametrizations are in $\mathfrak{p}$ after a certain finite number of coefficients are. 


	\section{Elliptic Curves}
	Given an elliptic curve $E$, 
	we denote the periods of $E$ by $\omega_1,\omega_2$, and the period lattice they generate by $\Lambda_E$. The Weierstrass $\wp$ function is defined in terms of $\Lambda_E$ and a complex variable $z$ as follows:	
	\[\wp(z,\Lambda_E) := \frac{1}{z^2} + \sum_{\substack{\lambda \in \Lambda_E \\ \lambda \neq 0}} \frac{1}{(z+\lambda)^2}-\frac{1}{\lambda^2}.\]
	
	The $\wp$-function \noindent $\wp(z,\Lambda_E)$ is even as a function of $z$, and its defining series is absolutely convergent and doubly periodic with periods $\omega_1,\omega_2$. The functions $\wp(z,\Lambda_E)$ and $\wp'(z,\Lambda_E)$ satisfy the relation
	\begin{equation}\label{DifEq}
	\wp'(z,\Lambda_E)^2 = 4\wp(z,\Lambda_E)^3 - g_2\wp(z,\Lambda_E)-g_3,
	\end{equation}
	
	\noindent where \[g_2 = g_2(\Lambda_E)  = 60 \sum_{\substack{\lambda \in \Lambda_E \\\lambda \neq 0}} (\lambda)^{-4} \]
	
	\noindent and \[g_3 = g_3(\Lambda_3) = 140 \sum_{\substack{\lambda \in \Lambda_E \\ \lambda \neq 0}} (\lambda)^{-6}.\]
	
	Also associated $E$ is the canonical differential 
	\[\omega = mf_E(z)dz,\]
	where $m$ is the Manin constant and $f_E$ is the weight two cusp form associated to $E$. 
	The Eichler integral is then defined as 
	
	\begin{equation}\label{defEichler}
	\varepsilon(z)=\int_z^{i\infty}\omega = \int_{z}^{i\infty}mf_E(\tau)d\tau.
	\end{equation}
	The function $\varepsilon(z)$ is not modular, but if 
	$\gamma = \left(\begin{smallmatrix}
	a & b \\
	c& d
	\end{smallmatrix}\right) \in \Gamma_0(N)$ acts as usual on the upper-half plane, then 
	
	\begin{align*}
	\frac{d}{dz} \big{(}\varepsilon(\gamma z)-\varepsilon(z)\big{)} &= \frac{d}{dz} 2\pi i\int_{\gamma z}^{z}mf_E(\tau)d\tau \\
	&=2\pi i m\big{(}f_E(z) - (cz+d)^2f_E(z)(cz+d)^{-2}\big{)} = 0
	\end{align*}
	
	\noindent where the second to last equality follows from the fundamental theorem of calculus and the modularity of $f_E$. So $\varepsilon(z)$ is \textit{almost} modular, in that the difference $\varepsilon(\gamma z)-\varepsilon(z)$ depends only on $\gamma$, and not on $z$. Denote this difference by  \[C(\gamma) := \varepsilon(\gamma z)-\varepsilon(z).\] One readily verifies that $C:\Gamma_0(N) \rightarrow m\Lambda_E$ is a group homomorphism. Eichler and Shimura \cite{Eichler,Shimura}  showed that when the Manin constant is $1$, that $C$ is actually an isomorphism. 
	
	For any $\lambda \in \cc$ such that $\lambda \in \textnormal{End}(E)$, we have that $\lambda \Lambda_E \subseteq \Lambda_E$. So it is possible to define 
		\[\wp_{\lambda}(z,\Lambda_E):= \lambda^2\wp(\lambda z,\Lambda_E) = \wp(z,\frac{1}{\lambda}\Lambda_E),\]
	\noindent where the extra factor $\lambda^2$ normalizes $\wp_{\lambda}$ to have a leading coefficient of $q^{-2}$ in its Fourier expansion. Similarly, 
	\[\wp_{\lambda}'(z,\Lambda_E):= \lambda^3\wp'(\lambda z,\Lambda_E) = \wp'(z,\frac{1}{\lambda}\Lambda_E).\]
	
	With this notation we define
		\[
	X_\lambda(z) =  m^2\wp_{\lambda}(\varepsilon(z),\Lambda_E) - \frac{a_1^2+4a_2}{12}, \]
	\[Y_\lambda(z) = \frac{m^3}{2}\wp_{\lambda}'(\varepsilon(z),\Lambda_E)-\frac{a_1 m^2}{2}\wp_{\lambda}(\varepsilon(z),\Lambda_E) +\frac{a_1^3+4a_1a_2-12a_3}{24}\]
	for $E$ given in general Weierstrass form with the convention that if the subscript $\lambda$ is omitted we take $\lambda = 1$. 
	Note that if $E$ is given in Wierstrass short form then 	\[X_\lambda(z) := m^2\wp_\lambda( \varepsilon(z),\Lambda_E) \quad Y_\lambda(z) := \frac{m^3}{2}\wp_\lambda'(\varepsilon(z),\Lambda_E).\]
	By construction $X_\lambda(z), Y_\lambda(z)$ satisfy the Wierstrass equation for the elliptic curve.
	Importantly, $X_\lambda(z)$ and $Y_\lambda(z)$ are modular over $\Gamma_0(N)$ since
	\[\wp_\lambda (\varepsilon(\gamma z),\Lambda_E) = \wp_\lambda( \varepsilon(z)+C(\gamma),\Lambda_E) = \wp_{\lambda}(\varepsilon(z),\Lambda_E)\]
	where the final equality holds because $\lambda C(\gamma) \in \Lambda_E$. A similar calculation holds for $Y_\lambda(z)$ as well as the parametrizations for the general form.

	\section{Expansions at Other Cusps}
	The first step in computing the coefficient functions $A_i$ in $(\ref{ModPoly})$ is to compute the $q$-expansions of each of the factors $(x-F(\gamma z))$ for $x$ a formal variable and $\gamma \in \SL_2(\zz)$.
	Since we are interested specifically in $F$ that are rational functions of $X_{\lambda}(z)$ and $Y_{\lambda}(z)$ it suffices to calculate the $q$-expansions for $X(\gamma z)$ and $Y(\gamma z)$.
	These coefficients are determined by two relations, 
	\begin{equation}\label{DifRel}
	qX' = (2Y + a_1X+a_3)f_E
	\end{equation}
	known as the invarient differential of $E$(see section \textrm{III} of  \cite{Silverman}), and the elliptic curve relation
	
	\begin{equation}
	Y^2 + a_1XY + a_3Y = X^3 + a_2X^2 +a_4X + a_6.
	\end{equation}
	
	\noindent A recursive algorithm was given by Cremona \cite{Cremona} using these two relations to calculate the expansions of $X(z)$ and $Y(z)$. Acting on $(\ref{defEichler})$ and $(\ref{DifRel})$ by a matrix $\gamma \in \SL_2(\zz)$ gives relations that allow us to recursively calculate the coefficients of modular parametrizations around cusps other than infinity. There are, however, a few complications we examine below.
	
	If we let $q_N(z) = e^{\frac{2\pi i }{N}z}$, we can write the expansions of the modular parametrizations at a cusp $\rho$ with width $w$ as $X_\lambda(\gamma z) = \sum_{n=-2}^\infty b_nq_w^n$ and $Y_\lambda(\gamma z) = \sum_{n=-3}^{\infty} d_nq_w^n$. Note that $b_i,d_i$ might be zero for $i =-3, -2,-1$ if neither $X$ nor $Y$ have poles at $\rho$. By examining the first few terms if the Laurent series of $\wp_{\lambda}$ and $\wp_\lambda'$ and evaluating them at $\varepsilon(\gamma z)$ we can calculate $b_{-2}$ and $d_{-3}$. So our inductive set up will be to assume that we know the $b_i$ coefficients for $-2\leq i \leq n-1$ and the $d_j$ coefficients for $-3\leq j \leq n-2$ and use this information to calculate $b_{n}$ and $d_{n-1}$. Letting $c_n$ denote the coefficient of $q_w^n$ of $f_E(\gamma z)$, relation $(\ref{defEichler})$ gives us that 
	
	\[\frac{1}{w}\sum_{n=-2}^{\infty}nb_{n}q_w^n = \Big{(} 
	2\sum_{n=-3}^{\infty}d_nq_w^n + a_1\sum_{n=-2}^{\infty}b_nq_w^n + a_3
	\Big{)}\sum_{n=1}^{\infty}c_nq_w^n.\]
	Comparing the coefficients of $q_w^n$ gives us one linear relation between $b_n$ and $d_{n-1}$
	
	\[nb_n = 2w\sum_{k=-3}^{n-1}c_{n-k}d_k + a_1w\sum_{k=-2}^{n-1}c_{n-k}b_{k}+a_3wc_n.\]
	
	\noindent Comparing the $q_w^{n-4}$ term in $(\ref{DifRel})$ gives us
	\begin{multline*}
	\sum_{k=-3}^{n-1}d_{n-4-k}d_k + a_1 \sum_{k=-3}^{n-4}b_{n-4-k}d_k + a_3d_{n-4} = \\
	\sum_{k = -2}^{n}\sum_{j=-2}^{n-2-k}b_{n-4-k-j}b_jb_k + a_2\sum_{k=-2}^{n-2}b_{n-4-j}b_j + a_4b_{n-4} + a_6^*
	\end{multline*}
	
	\noindent where $a_6^*$ indicates that this term is present only if $n-4 = 0$. This gives a second linear relation between $d_{n-1}$ and $b_{n}$, which allows us to solve for $d_{n-1}$ and $b_n$ uniquely whenever the determinant of the system is not $0$, i.e. when $-2nd_{-3}^2+6wc_1b_{-2}^2 \not=0$. Supposing that $X_\lambda(z)$  has a pole at $\rho$, (so that neither $d_{-3}$ nor $b_{-2}$ are $0$), then 
	
	\[-2n(d_{-3})^2+6wc_1(b_{-2})^2 = 0 \quad \implies n = \frac{3wc_1(b_{-2})^2}{(d_{-3})^2}.\]
	So this recursive process will not fail if we can find the first $\frac{3wc_1(b_{-2})^2}{(d_{-3}^2)}$ nontrivial terms of $X(z)$ and $Y(z)$ via the Laurent series expansions of $\wp_\lambda$ and $\wp_\lambda'$. Note that when $\rho = \infty$, we have that $w= c_1 = b_{-2} = d_{-3} = 1$ so that Cremona's algorithm doesn't fail with simply $3$ known terms of the Laurent expansion of $\wp_{\lambda}(\varepsilon(z))$.  
	
	However, if there are no poles at $\rho$, then $d_{-3} = b_{-2} = 0$, and the determinant will be $0$ for all $n$. So when calculating the $q_w$-expansions around cusps without poles, we need to compare other powers of $q_w$ to get information about such systems. Fortunately, we can simply compare powers of $q_{w}^n$ in $(\ref{defEichler})$ and $(\ref{DifRel})$ to get that a system with determinant $n(2d_0+a_1b_0+a_3)$.

	Interestingly, this determinant is zero when $2d_0+a_1b_0+a_3 = 0$, i.e when the constant terms of the expansions give a point of order $2$ on $E$. This is seen most easily by looking at $(\ref{defEichler})$, and observing that $2d_0+a_1b_0+a_3 = 0$ corresponds to a vertical tangent line on $E$. However, this is easily rectified. We first take $2d_0+a_1b_0+a_3 = 0$ as a hypothesis and compare powers of $q_w^{n}$ in $(\ref{defEichler})$ and powers of $q_{w}^{n}$ in $(\ref{DifRel})$ exactly like the previous case. The main difference is that since $2d_0+a_1b_0+a_3 = 0$, this gives us a system in the unknowns $b_n$ and $d_{n-1}$ instead of in terms of $b_n$ and $d_{n}$. So by examining $3$ cases we can effectively calculate the $q_w$-expansions of the modular parametrizations $X(z)$ and $Y(z)$ around any cusp. 
	
	Now that we can efficiently calculate these $q$-expansions for $X(\gamma z), Y(\gamma z)$ it is possible to construct 
	
	\[\Phi_F(x) := \prod_{\gamma \in \Gamma_0(N)\backslash SL_{2}(\zz)}\Big{(}x - F(\gamma z)\Big{)} = \sum A_i(z)x^{i}\]
	where $x$ is a formal variable and $F$ is any rational function in $X_{\lambda}(z)$ and $Y_{\lambda}(z)$. Note that by construction, the coefficients of $\Phi_F(x)$ are modular functions which are invariant under the action of $\SL_2(\zz)$, and so are rational functions in Klein's $j$-function. 
	
	In practice, in order to compute the minimal divisor of $\Phi_F(x)$ it is computationally advantageous to compute each of the functions $F(\gamma z)$ and then use symmetric polynomials to calculate the necessary coefficient functions until we locate all the poles of $F$.
\begin{Ex}
		Consider the elliptic curve 
		\begin{equation}
		E: y^2 + xy + y = x^3-x^2-3x+3. \tag{26b1}
		\end{equation}
		The point $(1,0)$ lies on $E$ and has $(1,-2)$ as its inverse. Then looking at the function $F(z) = \frac{Y(z)+2}{X(z)-1}$, we see that $F$ has a simple pole at the values $z \in \mathcal{H}$ that map $(X(z),Y(z))$ to  $(1,0)$. Note that the conductor of $E$ is $26$, and $[\SL_2(\zz):\Gamma_0(26)] = 42$. Calculating the trace of $\Phi_F$ (or the coefficient $A_{41}(z)$) we get 
		\[\sum_{\gamma \in \Gamma_0(26)\backslash \SL_2(\zz)}F(\gamma z) = \frac{-j(z)^2+54688 j(z) -37627200}{j(z) - 54000}.\]
		Testing the $42$ cosets of $\Gamma_0(26)$ in $\SL_2(\zz)$ gives us that for $z_0 = \frac{-7+\sqrt{-3}}{52}$, $\left(X(z_0),Y(z_0)\right) = (1,0)$. Thus the preimage of the rational point $(1,0)$ is a CM point on $X_0(26)$. 
	\end{Ex}
	
	Using this theory we are able to give a condition for when a point $P$ on an elliptic curve $E$ is the image of a CM point $\mathcal{P}$ on the modular curve and prove Theorem $1.2$.
	\begin{proof}
		Suppose that $m$ exactly divides $N$ and let $\mathcal{P}_2 = (e_2,c_2)$ be the image of $\mathcal{P}_1 = (e_1,c_1)$ under the Atkin-Lehner involution $W_m = \left(\begin{smallmatrix}
		am & b\\
		cN & dm
		\end{smallmatrix}\right)$ for integers $a,b,c,d$. The matrix $W_m$ imposes a rational map from $X_0(N)$ to itself, so if $e_1$ is not isomorphic to $e_2$, then $W_m$ is a rational isogeny of the curves $e_1$ and $e_2$. If $e_1$ is isomorphic to  $e_2$ and we write the periods for $e_1$, $e_2$ as $\omega_{11}, \omega_{12}$ and $\omega_{21}, \omega_{22}$ respectively, then $W_m$ takes $\tau_1 = \frac{\omega_{12}}{\omega_{11}}$ to $\tau_2 = \frac{\omega_{22}}{\omega_{21}}$. However, since $e_1 \cong e_2$, there must be a matrix $ A = \left(\begin{smallmatrix} 
		\alpha & \beta\\
		\gamma & \delta
		\end{smallmatrix}\right)$ in $\SL_2(\zz)$ such that $W_m \tau_1 = \tau_2 = A \tau_1$. This gives a quadratic relation that $\tau_1$ satisfies, namely 
		\[(am\tau_1 + b)(\gamma \tau_1 + \delta) = (\alpha\tau_1 + \beta)(cN\tau_1+dm).\]
		
		\noindent Expanding and collecting like terms gives 
		\[(am\gamma-c\alpha N)\tau_1^2 + (b\gamma+am\delta-cN\beta-dm\alpha)\tau_1+b\delta-dm\beta = 0.\]
		
		\noindent The discriminant of this quadratic is 
		\begin{align*}
		D &= (b\gamma +am\delta -cN\beta - dm\alpha)^2-4(am\gamma-c\alpha N)(b\delta-dm\beta) \\
		&= b^2\gamma^2 + a^2m^2\delta^2 + c^2N^2\beta^2+d^2m^2\alpha^2\\
		&+2b\gamma a m \delta - 2b\gamma cN\beta-2b\gamma d m \alpha -2am\delta c N \beta - 2adm^2\alpha \delta + 2cN\beta d m \alpha \\
		&-4( am\gamma b \delta - am^2d\beta \gamma - cNb\alpha \delta +  c\alpha N d m \beta).
		\end{align*}
		We collect like terms and use the fact that $\det(W_m) = adm^2-cNb = m$ to get	
		\begin{align*}
		D &= b^2\gamma^2 + a^2m^2\delta^2 + c^2N^2\beta^2+d^2m^2\alpha^2\\
		&-2b\gamma a m \delta + 2b\gamma cN\beta-2b\gamma d m \alpha -2am\delta c N \beta+ 2adm^2\alpha \delta- 2cN\beta d m \alpha \\
		&-4(m\alpha \delta-m\beta \gamma).
		\end{align*}
		Factoring and using that $\det(A) = \alpha\delta-\beta\gamma = 1$ gives that 
		\begin{align*}
		D &= (b\gamma -am\delta +cN\beta -dm\alpha)^2-4m.
		\end{align*}
		Thus $D$ is a square mod $4m$. Since $\tau_1$ is in the upper half plane, we must have that $D < 0$. However, since $(b\gamma -am\delta +cN\beta -dm\alpha)^2$ is non-negative, it follows that  $-4m \leq D < 0$.
	\end{proof}
	
	\begin{Ex}
		We return to the curve \begin{equation}
		E: y^2 + xy + y = x^3-x^2-3x+3 \tag{26b1}
		\end{equation}
		of conductor $26$ and index $42$. Consider the points $(1,-2)$ and $(3,2)$ with inverses $(1,0)$ and $(3,-6)$ on $E$. Then the functions $F$ and $G$ given by 
		\[F(z) = \frac{Y(z)-0}{X(z)-1}, \quad G(z) =  \frac{Y(z)+6}{X(z)-3}\]
		have simple poles for $z$ such that  $(X(z),Y(z)) =  (1,-2)$ or $(3,2)$ respectively. We calculate specific coefficient functions of $\Phi_{F} = \sum A_i(z)x^i$ and $\Phi_{G} = B_i(z)x^i$ to determine the location of these poles in the upper half plane:
		
		\begin{align*}
		A_{41}(z) &= \frac{-j(z)^2+288156\cdot j(z)-199626768}{j(z)-287496}, \\ B_{40}(z) &= \frac{j(z)^3 - 3214\cdot j(z)^2 + 2726620\cdot j - 274323456}{j(z) - 1728}.
		\end{align*}
		
		\noindent Thus $\Phi_F(z)$ has poles only when $j(z) = 287496$, i.e when $z$ is in the $\SL_2(\zz)$ orbit of $\sqrt{-4}$, and $G(z)$ has poles only when $j(z) = 1728$ i.e when $z$ is in the $\SL_2(\zz)$ orbit of $\sqrt{-1}$. Comparing the actions of the coset representatives of $\Gamma_0(26)$, we find that $z_0 := \frac{-5+\sqrt{-1}}{52}$ satisfies $(X(z),Y(z)) = (1,-2)$, and $z_1 = \frac{5+\sqrt{-1}}{13}$ satisfies $(X(z),Y(z)) = (3,2)$. 
		
		Examining the action of the Atkin-Lehner involutions $W_{2}$ and $W_{13}$, we find that $F_2 = F(W_2 z)$ , and $G_2 = G(W_2z)$  have coefficient functions

				\[A_{40}(z) = \frac{-j(z)^2+3235\cdot j(z) -2655936}{j(z)-1728}, \quad \quad
				B_{41}(z) =  \frac{-42\cdot j(z) + 21954240}{j(z)-287496}, \]
			while $F_{13} := F(W_{13}z)$ and $G_{13} := G(W_{13}z)$ have coefficient functions
			\begin{align*}
			A_{41}(z) &= \frac{-j(z)^2+288156\cdot j(z)-199626768}{j(z)-287496},\\
			B_{40}(z) &= \frac{j(z)^3 - 3214\cdot j(z)^2 + 2726620\cdot j - 274323456}{j(z) - 1728}.
			\end{align*}
\noindent Thus since $W_{2}$ exchanges the poles of $F$ and $G$, Theorem $1.2$ gives that the points $z_0$, $z_1$ correspond to isogenous elliptic curves on $X_0(26)$. Additionally, since $W_{13}$ fixes $z_0$ and $z_1$, Theorem $1.2$ also tells us they are both CM points on $X_0(26)$ whose orders have discriminants that must be squares mod $52$. In fact, the minimal polynomial of $z_0$ is $104z^2-20z+1$ which has discriminant $-16 \equiv 6^2 \mod 52$, and the minimal polynomial for $z_1$ is $13z^2-10z+2$ which has discriminant  $-4 \equiv 10^2 \mod 52$.
	\end{Ex}
\begin{Ex}
	Theorem $1.2$ can also be visualized in the following way. Consider again the elliptic curve $E: y^2 + y = x^3-x^2-10x-20$ of conductor $11$, and the fundamental domain $F_{11}$ in figure $1$ for the congruence subgroup $\Gamma_0(11)$.
	
	\begin{figure}[h]
		\centering
		\begin{minipage}{.5\textwidth}
			\centering
			\includegraphics[width=1\linewidth]{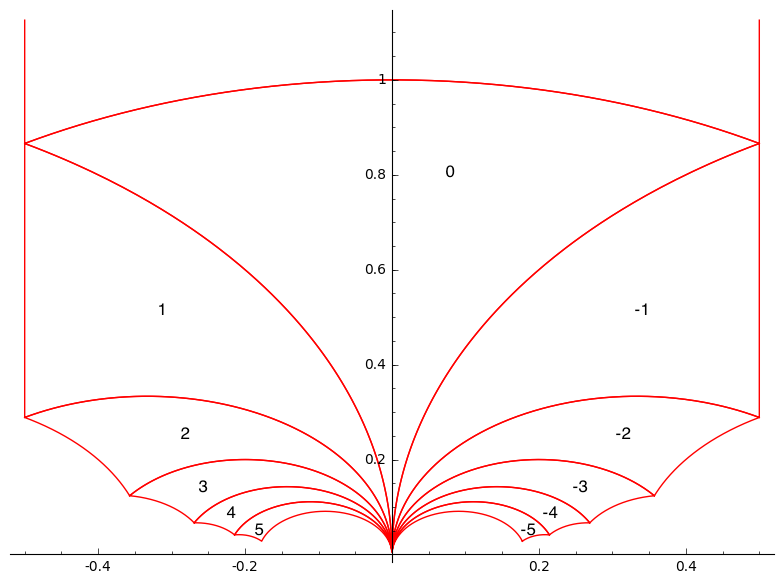}
			\captionof{figure}{fundamental domain $F_{11}$ for $\Gamma_0(11)$}
			\label{fig:test1}
		\end{minipage}%
		\begin{minipage}{.5\textwidth}
			\centering
			\includegraphics[width=1\linewidth]{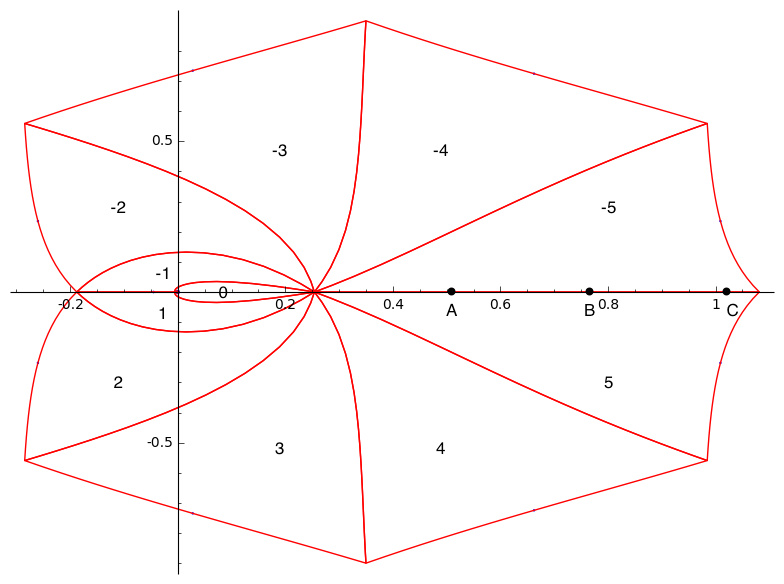}
			\captionof{figure}{Eichler integral over the boundary of $F_{11} $}
			\label{fig:test2}
		\end{minipage}
	\end{figure}

 This fundamental domain has been constructed by taking $\SL_2(\zz)$ coset representatives of the form $\left(\begin{smallmatrix}
	0 & -1 \\
	1 & j
	\end{smallmatrix}\right)$ for $-5\leq j \leq 5$, with each $j$ labeled in the corresponding hypertriangle. The associated newform of $E$ is $f_{E} = q-2q^2-q^3+2q^4\dots$. Taking complex values $z$ on the boundary of $F_{11}$ and calculating $\varepsilon(z) = \int_{z}^{i\infty}mf_E(\tau)d\tau$ gives the image in Figure $2$. The resulting image tiles the plane in a parallelogram-type pattern, with the same periods as $E$.
	The points $A,B$ and $C$ have been labeled at $2/5$, $3/5$ and $4/5$ times the real period of $E$ respectively. They correspond to the points $(5,-6), (5,5)$ and $(16,60)$ on $E$ respectively. The action of $W_{11}$ interchanges the two cusps in Figure $2$ ($\infty$ located at the origin, and $0$ located at the value $.2538\dots$ on the real line which is $1/5$ the real period of $E$). Up to translation by the real period, we see that $W_{11}$ interchanges the points $A$ and $C$ but fixes point $B$. By Theorem $1.2$ we conclude that the preimages of the points $(5,-6)$ and $(16,60)$ on $X_0(11)$ give isogenous elliptic curves, while the preimage of $(5,5)$ on $X_0(11)$ must be a CM point as we saw in Example $1.1$.
	\end{Ex}

	\section{Congruences Between Generated Algebras}
 Consider the elliptic curves $E_1$, $E_2$ given by
	\begin{align*}
	E_1: y^2 + xy + y &= x^3+4x-6, \tag{\textnormal{14a1}}\label{E14a1}\\
	E_2: y^2 + xy + y &= x^3 - 36x-70. \tag{\textnormal{14a2}}\label{E14a2}
	\end{align*}
	These curves have coefficients that are congruent mod $8$ and interestingly, if we look at the $q$-expansions of the row reduced basis elements of $\qq[X(z),Y(z)]$, we notice a similar phenomenon. \newline
	\begin{center}
	\begin{tabular}{|c|cccccc|}
		\hline
		\rule{0pt}{3ex}Basis over $E_1,$ $X = X_{E_1}(z)$, $Y = Y_{E_1}(z)$& \multicolumn{6}{c|}{$q$-expansion} \\ \hline
		\rule{0pt}{3ex}$1$&   &   &   & $1$  &  &   \\ \hline
		\rule{0pt}{3ex}$X(z)-2$&  $q^{-2}$ & $+q^{-1}$  &$+2q$   &$+2q^2$  &$+3q^3$  &$+\cdots  $\\ \hline
		\rule{0pt}{3ex}$-Y(z)-2X(z)-2$&  $q^{-3}$ &$+2q^{-1}$   &$+5q$   &$+4q^2$  &$+2q^3$&$+\cdots$  \\ \hline
		\rule{0pt}{3ex}$X(z)^2+2Y(z)-X(z)+2$&$q^{-4}$   &$-q^{-1}$   &$-2q$   &$+8q^2$  &$+5q^3$ &$+\cdots$  \\ \hline
		\rule{0pt}{3ex}$-Y(z)X(z)-3X(z)^2+2Y(z)+3X(z)-2$&$q^{-5}$   &   &$-2q$   &$-4q^2$  &$+18q^3$  &$+\cdots$  \\ \hline
		\rule{0pt}{3ex}$X(z)^3+3X(z)Y(z)-5Y(z)+2X(z)-6$ & $q^{-6}$ &$-2q^{-1}$ &$+4q$ &$ -7q^2$&$-6q^3$ &$+\cdots$ \\ \hline 
		\multicolumn{7}{}{} \\
		
		\hline
		\rule{0pt}{3ex}Basis over $E_2,$ $X = X_{E_2}(z)$, $Y = Y_{E_2}(z)$& \multicolumn{6}{c|}{$q$-expansion} \\ \hline
		\rule{0pt}{3ex}$1$&   &   &   & $1$  &  &   \\ \hline
		\rule{0pt}{3ex}$X(z)-2$&  $q^{-2}$ & $+q^{-1}$  &$+2q$   &$10q^2$  &$-5q^3$  &$+\cdots  $\\ \hline
		\rule{0pt}{3ex}$-Y(z)-2X(z)-2$&  $q^{-3}$ &$+2q^{-1}$   &$-3q$   &$-4q^2$  &$+2q^3$&$+\cdots$  \\ \hline
		\rule{0pt}{3ex}$X(z)^2+2Y(z)-X(z)-14$&$q^{-4}$   &$-q^{-1}$   &$+14q$   &  &$+29q^3$ &$+\cdots$  \\ \hline
		\rule{0pt}{3ex}$-Y(z)X(z)-3X(z)^2+2Y(z)+3X(z)+38$&$q^{-5}$   &   &$+6q$   &$-28q^2$  &$-14q^3$  &$+\cdots$  \\ \hline
		\rule{0pt}{3ex}$X(z)^3+3X(z)Y(z)-5Y(z)-22X(z)-6$ & $q^{-6}$ &$-2q^{-1}$ &$-12q$ &$ +25q^2$&$+138q^3$ &$+\cdots$ \\ \hline
		\multicolumn{7}{}{} \\
		
	\end{tabular} \newline
\end{center}

	\noindent The coefficients of the $q$-expansions are also congruent mod $8$.
	This is not simply a consequence of the congruence of the equations of $E_1$ and $E_2$. For example, the curves 
		\begin{align*}
E_3: y^2 + xy + y &= x^3+x^2-5x+2, \tag{\textnormal{15a3}}\\
E_4: y^2 + xy + y &= x^3 +x^2+35x-28. \tag{\textnormal{15a4}}
\end{align*}
	are congruent mod $10$, but the $q$ expansions of the $X$ term of their optimal modular parametrizations are
	\begin{align*}
	 X_{E_3}(z) &= q^{-2} + q^{-1} + 1 +2q + 3q^2+\hspace{6px} q^3 + \cdots  -6q^{11}+\cdots, \\
	X_{E_4}(z) &=  q^{-2}+q^{-1} +1 + 2q -5q^2+9q^3 + \cdots +7q^{11}+\cdots.
	\end{align*}
	Comparing the $q^{2}$ terms shows that any congruence between these two parametrizations must divide $8$, and comparing the $q^{11}$ terms shows that any such congruence must divide $13$. Thus we conclude that there are \textit{no} nontrivial congruences between the parametrizations. So when do congruences in the elliptic curve equation give rise to congruences in the generated algebras?

	If we assume that the two elliptic curves $E_1$ and $E_2$ given by 
	\begin{align*}
	E_1: y^2 + a_1xy + a_3y &= x^3 + a_2x^2 + a_4x+a_6,\\
	E_2: y^2 + \alpha_1xy + \alpha_3 y &= x^3 + \alpha_2x^2 + \alpha_4x + \alpha_6,
	\end{align*}
	are isogenous, then their period lattices will intersect nontrivially in a lattice $\Lambda_3$, corresponding to an elliptic curve $E_3$ with integral model \[y^2 +\beta_1xy + \beta_3y = x^3 + \beta_2x^2 + \beta_4 x + \beta_6.\] Thus the difference 
		\[g(z) :=\wp(z,\Lambda_1) - \wp(z,\Lambda_2)\] 
	
	\noindent is an even, elliptic function with period lattice $\Lambda_3$. If we let $\{r_i\}$ represent the complex numbers such that $\wp(r_i,\Lambda_3)$ is a zero of $g(z)$ in a fundamental parallelogram of $\Lambda_3$ and let $\{t_j\}$ be the values in $\Lambda_3$ such that $\wp(t_j,\Lambda_3)$ is a pole of $g(z)$ (repeated according to multiplicities) except possibly at the origin (even if the origin is a zero or pole of $g$), then the function 
	\[\frac{\prod_{i}\left(\wp(z,\Lambda_3)-\wp(r_i,\Lambda_3)\right)}{\prod_{j}(\wp(z,\Lambda_3)-\wp(t_j,\Lambda_3))}\] 
	
	\noindent is monic, and has the same zeros and poles as $g(z)$ except possibly at $0$. However, a classical arguement shows that the product must have the same zero or pole as $g(z)$ at $0$ as well (see \cite{koblitz} for example). 
\noindent Thus 
	\begin{equation}\label{ProducForm}
	g(z) = \wp(z,\Lambda_1) - \wp(z,\Lambda_2) = C\frac{\prod_{i}(\wp(z,\Lambda_3)-\wp(r_i,\Lambda_3))}{\prod_{j}(\wp(z,\Lambda_3)-\wp(t_j,\Lambda_3))}
	\end{equation}
	for some constant $C$. 	Since 
	\[\wp(z,\Lambda_1)-\wp(z,\Lambda_2) = \frac{g_2(\Lambda_1)-g_2(\Lambda_2)}{20}z^2 + \frac{g_3(\Lambda_1)-g_3(\Lambda_2)}{28}z^4+ \cdots\]
	
	\noindent we see that 
	
	\[C = C(\Lambda_1,\Lambda_2) = \begin{cases}
	\frac{g_2(\Lambda_1)-g_2(\Lambda_2)}{20} & \text{ if } g_2(\Lambda_1) \neq g_2(\Lambda_2)\\
	\frac{g_3(\Lambda_1)-g_3(\Lambda_2)}{28} & \text{ if } g_2(\Lambda_1) = g_2(\Lambda_2).
	\end{cases}\] 
	 With this notation we have the following. 
		\begin{Th}
		Suppose that $E_1,E_2$ are two isogenous elliptic curves over $\qq$. Also assume that the coordinates of the torsion points of order dividing $N$ in $\overline{\qq}$ are algebraic integers. Then there is an explicit natural number $D(\Lambda_1,\Lambda_2)$ so that the $q$-expansion of $X_{E_1}- X_{E_2}$ is congruent to a constant mod $C(\Lambda_1,\Lambda_2)/D(\Lambda_1,\Lambda_2)$.
	\end{Th}

	\begin{proof}
		Evaluating equation \eqref{ProducForm} at $\varepsilon(z)$, and adding the appropriate constant to both sides of the equality gives
	\begin{align*}
	X_{E_1}(z) - X_{E_2}(z) &= \wp(\varepsilon(z),\Lambda_1) + \frac{a_1^2-4a_2}{12} - \wp(\varepsilon(z),\Lambda_2) - \frac{\alpha_1^2-4\alpha_2}{12} \\ &=  C\frac{\prod_{i}(\wp(\varepsilon(z),\Lambda_3)-\wp(r_i,\Lambda_3))}{\prod_{j}(\wp(\varepsilon(z),\Lambda_3)-\wp(t_j,\Lambda_3))} + \frac{a_1^2-\alpha_1^2 +4\alpha_2 - 4a_2}{12}\\
	&= C \frac{\prod_i X_{E_3} - R_i}{\prod_j X_{E_3} - T_j} + \frac{a_1^2-\alpha_1^2 +4\alpha_2 - 4a_2}{12}\\
	\end{align*}
	where $R_i = \wp(r_i,\Lambda_3)-\frac{\beta_1^2-4\beta_2}{12}$ and $T_j = \wp(t_j,\Lambda_3)-\frac{\beta_1^2-4\beta_2}{12}$. The final equality follows from In fact that $X_{E_3} = \wp(z,\Lambda_3)+\frac{\beta_{1}^2-4\beta_4}{12}$ so that the fraction cancels out of the $X_{E_3}$ term and the $R_i$ or $T_j$ term. 
	
	The $T_j$'s are $x$-coordinates of torsion points of order dividing $N$ because the poles of $g(z)$ occur at lattice points of either $\Lambda_1$ or $\Lambda_2$. By hypothesis, these coordinates are algebraic integers. Since the $q$-expansions of both $X_{E_1}$ and $X_{E_2}$ are both integers, we also have that each of $\wp(r_i,\Lambda_3)$ must be algebraic. So we define $D = D(\Lambda_1,\Lambda_2) = \prod_i D_i$ where $D_i$ is the minimal natural number so that $D_iR_i$ is an algebraic integer.Thus 
	\[X_{E_1}(z)-X_{E_2}(z) = \frac{C}{D} \frac{\prod_{i}D_iX_{E_3}-D_iR_i}{\prod_jX_{E_3}-T_j}.\] Since the formal product $(\prod_j X_{E_3} - T_j)^{-1}$ has algebraic integer coefficients, and since $D_iR_i$ is an algebraic integer for all $i$, the above shows that all but the constant term of the $q$-expansion of $X_{E_1}(z)-X_{E_2}(z)$ are congruent to zero mod $C/D$. 
	\end{proof}
	\begin{Ex}
		Let's return to the curves $E_1$, $E_2$  (Cremona labels \ref{E14a1} and \ref{E14a2}) where we found a congruence mod $8$ between the $q$-expansions for their modular parametrizations. The period lattices for $E_1,E_2$ are given by the generators 
		$$(z_{11},z_{12}) \approx (1.981341, .990670 +1.325491 i),\quad (z_{21},z_{22}) \approx (.990670, 1.325491 i),$$
		
		\noindent and so we see that $\Lambda_{E_1} \subseteq \Lambda_{E_2}$. So we can write $\wp(z,\Lambda_2)$ as a rational function in $\wp(z,\Lambda_1)$. A quick calculation shows that in fact,
		\[\wp(z,\Lambda_1)-\wp(z,\Lambda_2) = \frac{8}{13/12-\wp(z,\Lambda_1)}.\]
		Since $X_{E_1}(z) = \wp(\varepsilon(z),\Lambda_1) -1/12$, we conclude that 
	\[ X_{E_1}(z) - X_{E_2}(z) = \frac{8}{1-X_{E_1}}.\]
	Since $X_{E_1}$ has integer coefficients, this makes the congruence mod $8$ between $X_{E_1}$ and $X_{E_2}$ now apparent. 
 \end{Ex}
\begin{Ex}
	Using Theorem $4.1$ we can now see why the curves 
		\begin{align*}
	E_3: y^2 + xy + y &= x^3+x^2-5x+2 \tag{\textnormal{15a3}},\\
	E_4: y^2 + xy + y &= x^3 +x^2+35x-28. \tag{\textnormal{15a4}},
	\end{align*}
	had only the trivial congruence mod $1$ even though their expressions share a congruence mod $10$.
	These curves are isogenous and $\Lambda_3 \subseteq \Lambda_4$, so we can write the difference $X_{E_4}-X_{E_3}$ as a rational funtion in terms of $X_{E_3}$. Since $g_2(\Lambda_{E_3})/20 = 241/240$ and $g_2(\Lambda_{E_4})/20 = -1679/240$, we see that $C = (241+1679)/240 = 8$. Also, we compute that
	\begin{align*}
	X_{E_4}-X_{E_3} =C \frac{-(X_{E_3}-\frac{3}{4})(X_{E_3}-\frac{3}{2})}{(X_{E_3}-1)(X_{E_3})^2}.
	\end{align*}
	So we see that $D = 8$ as well. Thus $C/D = 1$.
	\end{Ex}
	While Theorem $4.1$ describes many congruent algebras, it does not describe all congruences that we noticed computationally on curves of conductor less than $100$. For example, the curves 
	
	\begin{align*}
	E_1: y^2 &= x^3+x^2-32x+60 \tag{\textnormal{96a3}}\\
	E_2: y^2  &= x^3 + x^2 -384x +2772. \tag{\textnormal{48a5}}
	\end{align*}
	are not isogenous over $\qq$, so Theorem $4.1$ doesn't tell us of any congruences between the two algebras. However, looking at the difference of the $q$-expansions of the modular parametrizations of the $x$ coordinates of these two curves gives 
	\[-68q + 780q^3 - 5020q^5 + 24140q^7 - 96712q^9 + 340500q^{11} - 1086568q^{13}+\textnormal{O}(q^{15}).\]
	
	\noindent So we see that this form appears to be $0$ mod $4$. In fact, computationally we can confirm that a large number of coefficients are divisible by $4$. We would like to be able to tell that all of the coefficients are congruent to $0$ by looking at some finite number of terms in the $q$-expansion. To this end, we give a generalization of Sturm's bound that applies to meromorphic modular forms. The arguement is essentially the same, but we give a proof for completeness. For a modular form with $q$-expansion $f = \sum a_n q^n$ we denote 
	\[\ord_{\mathfrak{p}} f := \ord_{\infty}(f\bmod \mathfrak{p}) = \min\{n : a_n \not \in \mathfrak{p}\}\]
	and observe that since $\mathfrak{p}$ is a prime ideal, $\ord_\mathfrak{p}(fg) = \ord_\mathfrak{p}(f)+\ord_\mathfrak{p}(g)$.  
	We also denote by $M_k^{!!}(\Gamma,\mathcal{O})$ the collection of meromorphic modular forms of weight $k$ over $\Gamma$ with coefficients in $\mathcal{O}$. Finally, let $f^{[\gamma]_k}$ denote $(cz+d)^{-k}f(\gamma z)$ where $\gamma = \left(\begin{smallmatrix}
	a & b \\
	c&d \\
	\end{smallmatrix} \right)\in SL_2(\zz)$. With this notation we prove the following.
	\begin{Lemma}
		Let $\mathfrak{p}$ be a prime ideal in the ring of integers $\mathcal{O}$ of a number field $K$. Further suppose that $f \in M_k^{!!}(\Gamma,\mathcal{O})$ and $|\Gamma\backslash \SL_2(\zz)| = m$. Finally, let $\Omega$ be the set of points on $X_0(N)$ where $f$ has poles. Then 
		\[\textnormal{ord}_{\mathfrak{p}}(f) + \sum_{\tau \in \Omega} \textnormal{ord}_{\tau}(f) > \frac{km}{12}\]
		implies that $f\equiv 0 \pmod{\mathfrak{p}}$.
		
	\end{Lemma} 
	\begin{proof}
		We start with the case $\Gamma = \textnormal{SL}_2(\zz)$. We first note that since $f$ is meromorphic, $\textnormal{ord}_\tau f < \infty$ for all $\tau \in \Omega$. Also, since the coefficients of $f$ are elements of $\mathcal{O}$, for each of the finite complex numbers $\tau_i \in \Omega \cap \Gamma\backslash\mathcal{H}$, we can pick relatively prime algebraic integers $\alpha_i$, $\beta_i$ so that $\beta_ij(z)-\alpha_i$ has a zero of order at least $1$ at $\tau_i$. So 
		\[g(z) := f(z) \prod_{i}(\beta_ij(z)-\alpha_i)^{-\ord_{\tau_i}f}\]
		has poles only at infinity, and is modular over $\textnormal{SL}_2(\zz)$. Thus Sturm's theorem applies giving $g(z) \equiv 0 \mod \mathfrak{p}$ since 
		\begin{align*}
		\textnormal{ord}_{\pideal}(g) &= \textnormal{ord}_{\pideal}(f) - \sum_{\tau_i \in \Omega}\ord_{\tau_i}(f_i)\textnormal{ord}_{\pideal} (\beta_i j +\alpha_i) \\
		&\geq \textnormal{ord}_{\pideal}(f) + \sum_{\tau_i \in \Omega} \textnormal{ord}_{\tau_i}(f) > \frac{k}{12}.
		\end{align*} 
		
		\noindent The first inequality holds since $\alpha_i$ and $\beta_i$ are relatively prime algebraic integers in $\mathcal{O}$, implies that each of the terms $(\beta_i j +\alpha_i)$ has order $0, -1$ mod $\mathfrak{p}$ corresponding to $\beta_i \in \pp$ or not. Thus
		$g \equiv 0 \pmod{\mathfrak{p}}$ which implies that $f\equiv 0 \pmod{\pp}$. 
		This concludes the proof in the case that $\Gamma = \textnormal{SL}_2(\zz)$.
		
		If $\Gamma$ is an arbitrary congruence subgroup, we first pick $N$ so that $\Gamma(N) \subseteq \Gamma$ with $m$ coset representatives $\gamma_\ell$ for $\Gamma(N)$ and we set $L = K(\zeta_N)$. Since $f \in M_k^{!!}(\Gamma(N),L)$ and $\Gamma(N)$ is a normal subgroup of $\SL_2(\zz)$, the functions  $f^{[\gamma_\ell]_k}$ are elements of $M_k^{!!}(\Gamma(N),L)$. Furthermore, the denominators of the fourier coefficients of $f^{[\gamma_\ell]_k}$ are bounded because each is a finite $L$-linear combination of some integral basis of a finite dimensional subspace of $M_k^{!!}(\Gamma(N),L)$. Note that in general $M^{!!}_k(\Gamma(N),L)$ is not finite dimensional; however, if we restrict ourselves to the subspace that has poles of the same order and at the same locations as those of $f$ and $f^{[\gamma_\ell]_k}$, then this subspace is finite dimensional. Thus we can pick constants $A_\ell \in L^{\times}$ so that each of the functions $\ord_{\mathfrak{P}}(A_\ell f^{[\gamma_\ell]_k}) = 0$ for some prime ideal $\mathfrak{P}$ lying over $\pideal$. Letting $\gamma_1$ be the identity matrix, the function 
		\[G(z) := f(z) \prod_{\ell=2}^m A_\ell f^{[\gamma_\ell]_k}\]
		is a meromorphic modular form of weight $km$ over $\textnormal{SL}_2(\zz)$ with coefficients in $\mathcal{O}_L$. Then  \begin{align*}
		\ord_{\Pideal}(G) \geq  \ord_{\pideal}(G) \geq \ord_{\pideal}(f) + \sum_{\tau \in \Omega} \ord_\tau(f) > \frac{km}{12},
		\end{align*}
		where the first equality follows because $\mathfrak{P} \cap \mathcal{O}_K = \pideal$.
		We conclude that $G \equiv 0 \pmod{\mathfrak{P}}$ from the $\SL_2(\zz)$ case. Since each of the functions $A_{\gamma_\ell}f^{[\gamma_\ell]_k}$ were chosen such that $\ord_{\mathfrak{P}}(A_\ell f^{[\gamma_\ell]_k}) = 0$, this gives $G \equiv 0 \pmod{\pp}$ and so $f \equiv 0 \pmod{\pp}$. See theorem $9.18$ in \cite{Stein} to compare the above to the proof of Sturm's theorem for elements of $M_k(\Gamma,\mathcal{O})$.
	\end{proof}
	
	\begin{Cor}
		If $X_{E_1}$ and $X_{E_2}$ are modular parametrizations for the $x$ coordiantes of elliptic curves $E_1$ and $E_2$ of conductor $N_1$ and $N_2$ with modular degrees $d_1$ and $d_2$ respectively, then if $\ord_{p}(X_{E_1}-X_{E_2}) > 2(d_1+d_2)$, then $X_{E_1}\equiv X_{E_2} \mod p$.
	\end{Cor}
	\begin{proof}
		The number of poles of $X_{E_i}$ is at most $2d_i$ counting multiplicities. Thus the corollary follows immediately from Theorem $4.4$ applied to the difference $X_{E_1} - X_{E_2}$ which is modular over $\Gamma_0(\text{lcm}\left(N_1,N_2\right))$ since 
		\[\ord_p(X_{E_1}-X_{E_2}) + \sum_{\tau \in \omega}\ord_\tau(X_{E_1}-X_{E_2}) > 2(d_1+d_2)-2(d_1+d_2) = 0 = \frac{km}{12}.\]
			\end{proof}
	 Note that this bound is independent of both $N_1$ and $N_2$ since the weight $k$ of the modular parametrizations is zero. We obtain a better estimate if we know a priori the locations of the poles of $X_{E_i}$ and if they cancel in the difference $X_{E_1}-X_{E_2}$.

	Corollary $4.4$ gives us an easy way for determining if two related parametrizations are congruent mod $\pp$. Returning to our earlier example with the curves  \begin{align*}
	E_1: y^2 &= x^3+x^2-32x+60 \tag{\textnormal{96a3}},\\
	E_2: y^2  &= x^3 + x^2 -384x +2772 \tag{\textnormal{48a5}},
	\end{align*}
	since the modular degree of both $E_1$ and $E_2$ is $8$, computing $2(8+8) = 32$ coefficients of the difference function and observing that they are congruent to $0$ mod $4$ is sufficient to prove that all of the coefficients are congruent mod $4$. 

	\bibliographystyle{plain}
	\bibliography{ModularParametrizations}

\end{document}